\documentclass{amsart}

\usepackage[utf8]{inputenc}
\usepackage{amsmath,amssymb}
\usepackage[autolanguage]{numprint}
\usepackage{amsthm}
\usepackage{hyperref}
\usepackage[capitalize]{cleveref}
\usepackage[square,numbers,sort]{natbib}
\usepackage{listings}[2013/08/05]
\hypersetup{pdfstartview=}

\newcommand\BZ{\mathbb Z}
\newcommand\BN{\mathbb N}

\def\BHM#1.#2.#3.#4.{{^{#1}_{#3}\mathcal B^{#2}_{#4}}}

\newcommand\comm\curlyvee
\newcommand\cocomm\curlywedge

\newcommand\inv{^{-1}}

\DeclareMathOperator{\class}{class}

\theoremstyle{plain}
\newtheorem{thm}{Theorem}[section]
\newtheorem{cor}[thm]{Corollary}

\newtheorem{lem}[thm]{Lemma}

\theoremstyle{definition}
\newtheorem{df}[thm]{Definition}
\newtheorem{example}[thm]{Example}

\theoremstyle{remark}
\newtheorem{rem}[thm]{Remark}
\newtheorem{conj}[thm]{Conjecture}

\crefname{lem}{Lemma}{Lemmas}
\crefname{thm}{Theorem}{Theorems}
\crefname{cor}{Corollary}{Corollaries}
\crefname{prop}{Proposition}{Propositions}
\crefname{example}{example}{examples}
\crefname{df}{Definition}{Definitions}
\crefname{equation}{equation}{equations}

\numberwithin{equation}{section}

\renewcommand\iff{\Leftrightarrow}

\def\clap#1{\hbox to 0pt{\hss#1\hss}}

\title{The FSZ properties of sporadic simple groups}
\author{Marc Keilberg}
\begin{document}
\begin{abstract}
We investigate a possible connection between the $FSZ$ properties of a group and its Sylow subgroups.  We show that the simple groups $G_2(5)$ and $S_6(5)$, as well as all sporadic simple groups with order divisible by $5^6$ are not $FSZ$, and that neither are their Sylow 5-subgroups.  The groups $G_2(5)$ and $HN$ were previously established as non-$FSZ$ by Peter Schauenburg; we present alternative proofs.  All other sporadic simple groups and their Sylow subgroups are shown to be $FSZ$.  We conclude by considering all perfect groups available through GAP with order at most $10^6$, and show they are non-$FSZ$ if and only if their Sylow 5-subgroups are non-$FSZ$.
\end{abstract}
\email{keilberg@usc.edu}
\thanks{This work is in part an outgrowth of an extended e-mail discussion between Geoff Mason, Susan Montgomery, Peter Schauenburg, Miodrag Iovanov, and the author.  The author thanks everyone involved for their contributions, feedback, and encouragement.}
\maketitle

\section*{Introduction}
The $FSZ$ properties for groups, as introduced by \citet{IMM}, arise from considerations of certain invariants of the representation categories of semisimple Hopf algebras known as higher Frobenius-Schur indicators \citep{KSZ2,NS08,NS07b}.  See \citep{NegNg16} for a detailed discussion of the many important uses and generalizations of these invariants.  When applied to Drinfeld doubles of finite groups, these invariants are described entirely in group theoretical terms, and are in particular invariants of the group itself.  The $FSZ$ property is then concerned with whether or not these invariants are always integers---which gives the $Z$ in $FSZ$.

While the $FSZ$ and non-$FSZ$ group properties are well-behaved with respect to direct products \citep[Example 4.5]{IMM}, there is currently little reason to suspect a particularly strong connection to proper subgroups which are not direct factors.  Indeed, by \citep{IMM,Etingof:SymFSZ} the symmetric groups $S_n$ are $FSZ$, while there exist non-$FSZ$ groups of order $5^6$.  Therefore, $S_n$ is $FSZ$ but contains non-$FSZ$ subgroups for all sufficiently large $n$.  On the other hand, non-$FSZ$ groups can have every proper subquotient be $FSZ$.  Even the known connection to the one element centralizers---see the comment following \cref{df:gm-sets}---is relatively weak.  In this paper we will establish a few simple improvements to this situation, and then proceed to establish a number of examples of $FSZ$ and non-$FSZ$ groups that support a potential connection to Sylow subgroups.  We propose this connection as \cref{conj}.

We will make extensive use of \citet{GAP4.8.4} and the \verb"AtlasRep"\cite{AtlasRep} package.  Most of the calculations were designed to be completed with only 2GB of memory or (much) less available---in particular, using only a 32-bit implementation of GAP---, though in a few cases a larger workspace was necessary.  In all cases the calculations can be completed in workspaces with no more than 10GB of memory available.  The author ran the code on an {\tt Intel(R) Core(TM) i7-4770 CPU @ 3.40GHz} machine with 12GB of memory.  All statements about runtime are made with respect to this computer.  Most of the calculations dealing with a particular group were completed in a matter of minutes or less, though calculations that involve checking large numbers of groups can take several days or more across multiple processors.

The structure of the paper is as follows.  We introduce the relevant notation, definitions, and background information in \cref{sec:notation}.  In \cref{sec:sub} we present a few simple results which offer some connections between the $FSZ$ (or non-$FSZ$) property of $G$ and certain of its subgroups.  This motivates the principle investigation of the rest of the paper: comparing the $FSZ$ properties for certain groups and their Sylow subgroups.  In \cref{sec:functions} we introduce the core functions we will need to perform our calculations in GAP.  We also show that all groups of order less than 2016 (except possibly those of order 1024) are $FSZ$.  The remainder of the paper will be dedicated to exhibiting a number of examples that support \cref{conj}.

In \cref{sec:simple} we show that the simple groups $G_2(5)$, $HN$, $Ly$, $B$, and $M$, as well as their Sylow 5-subgroups, are all non-$FSZ_5$.  In \cref{sec:fischer} we show that all other sporadic simple groups (including the Tits group) and their Sylow subgroups are $FSZ$.  This is summarized in \cref{thm:summary}.  The case of the simple projective symplectic group $S_6(5)$ is handled in \cref{sec:symplectic}, which establishes $S_6(5)$ as the second smallest non-$FSZ$ simple group after $G_2(5)$.  It follows from the investigations of \citet{PS16} that $HN$ is then the third smallest non-$FSZ$ simple group.  $S_6(5)$ was not susceptible to the methods of \citet{PS16}, and requires further modifications to our own methods to complete in reasonable time.  We finish our examples in \cref{sec:perfect} by examining those perfect groups available through GAP, and show that they are $FSZ$ if and only if their Sylow subgroups are $FSZ$. Indeed, they are non-$FSZ$ if and only if their Sylow 5-subgroup is non-$FSZ_5$.

Of necessity, these results also establish that various centralizers and maximal subgroups in the groups in question are also non-$FSZ_5$, which can be taken as additional examples.  If the reader is interested in $FSZ$ properties for other simple groups, we note that \citet{PS16} has checked all simple groups of order at most $|HN| = \numprint{273030912000000} = 2^{14} \cdot 3^6\cdot 5^6 \cdot 7 \cdot 11 \cdot 19$, except for $S_6(5)$ (which we resolve here); and that several families of simple groups were established as $FSZ$ by \citet{IMM}.

We caution the reader that the constant recurrence of the number $5$ and Sylow 5-subgroups of order $5^6$ in this paper is currently more of a computationally convenient coincidence than anything else.  The reasons for this will be mentioned during the course of the paper.

\section{Background and Notation}\label{sec:notation}
Let $\BN$ be the set of positive integers.  The study of $FSZ$ groups is connected to the following sets.
\begin{df}\label{df:gm-sets}
  Let $G$ be a group, $u,g\in G$, and $m\in \BN$.  Then we define
  \[ G_m(u,g) = \{ a\in G \ : \ a^m = (au\inv)^m = g\}.\]
\end{df}
Note that $G_m(u,g)=\emptyset$ if $u\not\in C_G(g)$, and that in all cases $G_m(u,g)\subseteq C_G(g)$.  In particular, letting $H=C_G(g)$, then when $u\in H$ we have
\[ G_m(u,g)=H_m(u,g).\]

The following will then serve as our definition of the $FSZ_m$ property.  It's equivalence to other definitions follows easily from \citep[Corollary 3.2]{IMM} and applications of the Chinese remainder theorem.
\begin{df}\label{df:FSZ}
  A group $G$ is $FSZ_m$ if and only if for all $g\in G$, $u\in C_G(g)$, and $n\in\BN$ coprime to the order of $g$, we have
  \[ |G_m(u,g)| = |G_m(u,g^n)|.\] We say a group is $FSZ$ if it is $FSZ_m$ for all $m$.
\end{df}
The following result is useful for reducing the investigation of the $FSZ$ properties to the level of conjugacy classes or even rational classes.
\begin{lem}\label{cor:u-g-exchange}
  For any group $G$ and $u,g,x\in G$ we have a bijection $G_m(u,g)\to G_m(u^x,g^x)$ given by $a\mapsto a^x$.

  If $n\in\BN$ is coprime to $|G|$ and $r\in\BN$ is such that $rn\equiv 1\bmod |G|$, we also have a bijection $G_m(u,g^n)\to G_m(u^r,g)$.
\end{lem}
\begin{proof}
  The first part is \citep[Proposition 7.2]{KSZ2} in slightly different notation.  The second part is \citep[Corollary 5.5]{PS:Quasitensor}.
\end{proof}

All expressions of the form $G_m(u,g^n)$ will implicitly assume that $n$ is coprime to the order of $g$.  We are free to replace $n$ with an equivalent value which is coprime to $|G|$ whenever necessary.  Moreover, when computing cardinalities $|G_m(u,g)|$ it suffices to compute the cardinalities $|H_m(u,g)|$ for $H=C_G(g)$, instead.  This latter fact is very useful when attempting to work with groups of large order, or groups with centralizers that are easy to compute in, especially when the group is suspected of being non-$FSZ$.

\begin{rem}
  There are stronger conditions called $FSZ_m^+$, the union of which yields the $FSZ^+$ condition, which are also introduced by \citet{IMM}.  The $FSZ_m^+$ condition is equivalent to the centralizer of every non-identity element with order not in $\{1,2,3,4,6\}$ being $FSZ_m$, which is in turn equivalent to the sets $G_m(u,g)$ and $G_m(u,g^n)$ being isomorphic permutation modules for the two element centralizer $C_G(u,g)$ \citep[Theorem 3.8]{IMM},  with $u,g,n$ satisfying the same constraints as for the $FSZ_m$ property.  Here the action is by conjugation.  We note that while the $FSZ$ property is concerned with certain invariants being in $\BZ$, the $FSZ^+$ property is not concerned with these invariants being non-negative integers.  When the invariants are guaranteed to be non-negative is another area of research, and will also not be considered here.
\end{rem}

\begin{example}
  The author has shown that quaternion groups and certain semidirect products defined from cyclic groups are always $FSZ$ \citep{K,K2}.  This includes the dihedral groups, semidihedral groups, and quasidihedral groups, among many others.
\end{example}
\begin{example}
  \citet{IMM} showed that several groups and families of groups are $FSZ$, including:
  \begin{itemize}
    \item All regular $p$-groups.
    \item $\BZ_p\wr_r\BZ_p$, the Sylow $p$-subgroup of $S_{p^2}$, which is an irregular $FSZ$ $p$-group.
    \item $PSL_2(q)$ for a prime power $q$.
    \item Any direct product of $FSZ$ groups.  Indeed, any direct product of $FSZ_m$ groups is also $FSZ_m$, as the cardinalities of the sets in \cref{df:gm-sets} split over the direct product in an obvious fashion.
    \item The Mathieu groups $M_{11}$ and $M_{12}$.
    \item Symmetric and alternating groups.  See also \citep{Etingof:SymFSZ}.
  \end{itemize}
  Because of the first item, Susan Montgomery has proposed that we use the term $FS$-regular instead of $FSZ$, and $FS$-irregular for non-$FSZ$.  Similarly for $FS_m$-regular and $FS_m$-irregular.  These seem reasonable choices, but for this paper the author will stick with the existing terminology.
\end{example}
\begin{example}
  On the other hand, \citet{IMM} also established that non-$FSZ$ groups exist by using \citet{GAP4.8.4} to show that there are exactly 32 isomorphism classes of groups of order $5^6$ which are not $FSZ_5$.
\end{example}
\begin{example}
  The author has constructed examples of non-$FSZ_{p^j}$ $p$-groups for all primes $p>3$ and $j\in\BN$ in \citep{K16:p-examples}.  For $j=1$ these groups have order $p^{p+1}$, which is the minimum possible order for any non-$FSZ$ $p$-group.  Combined, \citep{IMM,K16:p-examples,PS16} show, among other things, that the minimum order of non-$FSZ$ 2-groups is at least $2^{10}$, and the minimum order for non-$FSZ$ 3-groups is at least $3^8$.  It is unknown if any non-$FSZ$ 2-groups or 3-groups exist, however.
\end{example}
\begin{example}
  \citet{PS16} provides several equivalent formulations of the $FSZ_m$ properties, and uses them to construct \citet{GAP4.8.4} functions which are useful for testing the property.  Using these functions, it was shown that the Chevalley group $G_2(5)$ and the sporadic simple group $HN$ are not $FSZ_5$.  These groups were attacked directly, using advanced computing resources for $HN$, often with an eye on computing the values of the indicators explicitly.  We will later present an alternative way of using GAP to prove that these groups, and their Sylow 5-subgroups, are not $FSZ_5$.  We will not attempt to compute the actual values of the indicators, however.
\end{example}

One consequence of these examples is that the smallest known order for a non-$FSZ$ group is $5^6=\numprint{15625}$.  The groups with order divisible by $p^{p+1}$ for $p>5$ that are readily available through GAP are small in number, problematically large, and frequently do not have convenient representations.  Matrix groups have so far proven too memory intensive for what we need to do, so we need permutation or polycyclic presentations for accessible calculations.  For these reasons, all of the examples we pursue in the following sections will hone in on the non-$FSZ_5$ property for groups with order divisible by $5^6$, and which admit known or reasonably computable permutation representations.  In most of the examples, $5^6$ is the largest power of $5$ dividing the order, with the Monster group, the projective symplectic group $S_6(5)$, and the perfect groups of order $12\cdot 5^7$ being the exceptions.

\section{Obtaining the non-\ensuremath{FSZ} property from certain subgroups}\label{sec:sub}
Our first elementary result offers a starting point for investigating non-$FSZ_m$ groups of minimal order.
\begin{lem}\label{lem:minorder}
  Let $G$ be a group with minimal order in the class of non-$FSZ_{m}$ groups. Then $|G_{m}(u,g)|\neq |G_{m}(u,g^n)|$ for some $(n,|G|)=1$ implies $g\in Z(G)$.
\end{lem}
\begin{proof}
  If not then $C_G(g)$ is a smaller non-$FSZ_{m}$ group, a contradiction.
\end{proof}
The result applies to non-$FSZ_m$ groups in a class that is suitably closed under the taking of centralizers.  For example, we have the following version for $p$-groups.
\begin{cor}\label{cor:minorder}
  Let $P$ be a $p$-group with minimal order in the class of non-$FSZ_{p^j}$ $p$-groups. Then $|P_{p^j}(u,g)|\neq |P_{p^j}(u,g^n)|$ for some $p\nmid n$ implies $g\in Z(P)$.
\end{cor}

\begin{example}
  From the examples in the previous section, we know the minimum possible order for a non-$FSZ_p$ $p$-group for $p>3$ is $p^{p+1}$.  It remains unknown if the examples of non-$FSZ_{p^j}$ $p$-groups from \citep{K16:p-examples} for $j>1$ have minimal order among non-$FSZ_{p^j}$ $p$-groups.  We also know that to check if a group of order $2^{10}$ or $3^8$ is $FSZ$ it suffices to assume that $g$ is central.
\end{example}

Next, we determine a condition for when the non-$FSZ$ property for a normal subgroup implies the non-$FSZ$ property for the full group.

\begin{lem}
  Let $G$ be a group and suppose $H$ is a non-$FSZ_m$ normal subgroup with $m$ coprime to $[G:H]$.  Then $G$ is non-$FSZ_m$.
\end{lem}
\begin{proof}
  Let $u,g\in H$ and $(n,|g|)=1$ be such that $|H_m(u,g)|\neq |H_m(u,g^n)|$.  By the index assumption, for all $x\in G$ we have $x^m\in H \iff x\in H$, so by definitions $G_m(u,g)=H_m(u,g)$ and $G_m(u,g^n)=H_m(u,g^n)$, which gives the desired result.
\end{proof}

\begin{cor}\label{cor:normalSylow}
  Let $G$ be a finite group and suppose $P$ is a normal non-$FSZ_{p^j}$ Sylow $p$-subgroup of $G$ for some prime $p$.  Then $G$ is non-$FSZ_{p^j}$.
\end{cor}
\begin{cor}\label{cor:normalizer}
  Let $G$ be a finite group and $P$ a non-$FSZ_{p^j}$ Sylow $p$-subgroup of $G$.  Then the normalizer $N_G(P)$ is non-$FSZ_{p^j}$.
\end{cor}
Sadly, we will find no actual use for \cref{cor:normalSylow} in the examples we consider in this paper.  However, this result, \citep[Lemma 8.7]{PS16}, and the examples we collect in the remainder of this paper do suggest the following conjectural relation for the $FSZ$ property.
\begin{conj}\label{conj}
  A group is $FSZ$ if and only if all of its Sylow subgroups are $FSZ$.
\end{conj}

Some remarks on why this conjecture may involve some deep results to establish affirmatively seems in order.

Consider a group $G$ and let $u,g\in G$ and $n\in\BN$ with $(n,|G|)=1$.  Suppose that $g$ has order a power of $p$, for some prime $p$.  Then
\[ G_{p^j}(u,g) = \bigcup G_{p^j}(u,g)\cap P^x,\]
where the union runs over all distinct conjugates $P^x$ in $C_G(g)$ of a fixed Sylow $p$-subgroup $P$ of $C_G(g)$.  Let $P^x_{p^j}(u,g) = G_{p^j}(u,g)\cap P^x$.  Then $|G_{p^j}(u,g)|=|G_{p^j}(u,g^n)|$ if and only if there is a bijection $\bigcup P^x(u,g)\to \bigcup P^x(u,g^n)$.  In the special case $u\in P$, if $P$ was $FSZ_{p^j}$ we would have a bijection $P(u,g)\to P(u,g^n)$, but this does not obviously guarantee a bijection $P^x(u,g)\to P^x(u,g^n)$ for all conjugates.  Attempting to get a bijection $\bigcup P^x(u,g)\to \bigcup P^x(u,g^n)$ amounts, via the Inclusion-Exclusion Principle, to controlling the intersections of any number of conjugates and how many elements those intersections contribute to $G_{p^j}(u,g)$ and $G_{p^j}(u,g^n)$.  There is no easy or known way to predict the intersections of a collection of Sylow $p$-subgroups for a completely arbitrary $G$, so any positive affirmation of the conjecture will impose a certain constraint on these intersections.

Moreover, we have not considered the case of the sets $G_m(u,g)$ where $m$ has more than one prime divisor, nor those cases where $u,g$ do not order a power of a fixed prime, so a positive affirmation of the conjecture is also expected to show that the $FSZ_m$ properties are all derived from the $FSZ_{p^j}$ properties for all prime powers dividing $m$.  On the other hand, a counterexample seems likely to involve constructing a large group which exhibits a complex pattern of intersections in its Sylow $p$-subgroups for some prime $p$, or otherwise exhibits the first example of a group which is $FSZ_{p^j}$ for all prime powers but is nevertheless not $FSZ$.

\begin{example}
  All currently known non-$FSZ$ groups are either $p$-groups (for which the conjecture is trivial), are nilpotent (so are just direct products of their Sylow subgroups), or come from perfect groups (though the relevant centralizers need not be perfect).  The examples of both $FSZ$ and non-$FSZ$ groups we establish here will also all come from perfect groups and $p$-groups.  In the process we obtain, via the centralizers and maximal subgroups considered, an example of a solvable, non-nilpotent, non-$FSZ$ group; as well as an example of a non-$FSZ$ group which is neither perfect nor solvable.  All of these examples, of course, conform to the conjecture.
\end{example}

\section{GAP functions and groups of small order}\label{sec:functions}
The current gold standard for general purpose testing of the $FSZ$ properties in \citet{GAP4.8.4} is the \verb"FSZtest" function of \citet{PS16}.  In certain specific situations, the function \verb"FSInd" from \citep{IMM} can also be useful for showing a group is non-$FSZ$.  However, with most of the groups we will consider in this paper both of these functions are impractical to apply directly.  The principle obstruction for \verb"FSZtest" is that this function needs to compute both conjugacy classes and character tables of centralizers, and this can be a memory intensive if not wholly inaccessible task.  For \verb"FSInd" the primary obstruction, beyond its specialized usage case, is that it must completely enumerate, store, and sort the entire group (or centralizer).  This, too, can quickly run into issues with memory consumption.

We therefore need alternatives for testing (the failure of) the FSZ properties which can sidestep such memory consumption issues.  For \cref{sec:perfect} we will also desire functions which can help us detect and eliminate the more "obviously" FSZ groups.  We will further need to make various alterations to \verb"FSZtest" to incorporate these things, and to return a more useful value when the group is not $FSZ$.

The first function we need, \verb"FSZtestZ", is identical to \verb"FSZtest"---and uses several of the helper functions found in \citep{PS16}---except that instead of calculating and iterating over all rational classes of the group it iterates only over those of the center.  It needs only a single input, which is the group to be checked.  If it finds that the group is non-$FSZ$, rather than return \verb"false" it returns the data that established the non-$FSZ$ property.  Of particular importance are the values \verb"m" and \verb"z".  If the group is not shown to be non-$FSZ$ by this test, then it returns \verb"fail" to indicate that the test is typically inconclusive.

\begin{lstlisting}
FSZtestZ := function(G)
local CT, zz, z , cl, div, d, chi, m, b ;

cl := RationalClasses(Center(G));
cl := Filtered(cl,c->not Order(Representative(c))
                in [1,2,3,4,6]);

for zz in cl do
    z := Representative(zz);

	div := Filtered(DivisorsInt(Exponent(G)/Order(z)),
            m->not Gcd(m,Order(z)) in [1,2,3,4,6]);
	if Length(div) < 1 then continue; fi;

	CT := OrdinaryCharacterTable(G);

	for chi in Irr(CT) do
		for m in div do
			if not IsRat(beta(CT, z, m, chi))
                            then return [z,m,chi,CT];
			fi ;
		od;
	od;
od;

#the test is inconclusive in general
return fail ;

end ;
\end{lstlisting}
This function is primarily useful for testing groups with minimal order in a class closed under centralizers, such as in \cref{lem:minorder,cor:minorder}. Or for any group with non-trivial center that is suspected of failing the $FSZ$ property at a central value.

We next desire a function which can quickly eliminate certain types of groups as automatically being $FSZ$.  For this, the following result on groups of small order is helpful.
\begin{thm}\label{thm:small-order}
  Let $G$ be a group with $|G|<2016$ and $|G|\neq 1024$.  Then $G$ is $FSZ$.
\end{thm}
\begin{proof}
  By \cref{lem:minorder} it suffices to run \verb"FSZtestZ" over all groups in the \verb"SmallGroups" library of GAP.  This library includes all groups with $|G|<2016$, except those of order $2^{10}=1024$.  In practice, the author also used the function \verb"IMMtests" introduced below, but where the check on the size of the group is constrained initially to 100 by \citep[Corollary 5.5]{IMM}, and can be increased whenever desired to eliminate all groups of orders already completely tested.  This boils down to quickly eliminating $p$-groups and groups with relatively small exponent.  By using the closure of the $FSZ$ properties with respect to direct products, one need only consider a certain subset of the orders in question rather than every single one in turn, so as to avoid essentially double-checking groups.  We note that the groups of order 1536 take the longest to check.  The entire process takes several days over multiple processors, but is otherwise straightforward.
\end{proof}

We now define the function \verb"IMMtests". This function implements most of the more easily checked conditions found in \citep{IMM} that guarantee the $FSZ$ property, and calls \verb"FSZtestZ" when it encounters a suitable $p$-group.  The function returns \verb"true" if the test conclusively establishes that the group is $FSZ$; the return value of \verb"FSZtestZ" if it conclusively determines the group is non-$FSZ$; and \verb"fail" otherwise.  Note that whenever this function calls \verb"FSZtestZ" that test is conclusive by \cref{cor:minorder}, so it must adjust a return value of \verb"fail" to \verb"true".
\begin{lstlisting}
IMMtests := function(G)
	local sz, b, l, p2, p3, po;

	if IsAbelian(G)
		then return true;
	fi;
	
	sz := Size(G);

	if (sz < 2016) and (not sz=1024)
		then return true;
	fi;
		
	if IsPGroup(G) then
        #Regular p-groups are always FSZ.

		l := Collected(FactorsInt(sz))[1];

		if l[1]>= l[2] or Exponent(G) = l[1]
			then return true;
		fi;

		sz := Length(UpperCentralSeries(G));

		if l[1]=2 then
			if l[2]<10 or sz < 3
                         or Exponent(G)<64
				then return true;
			elif l[2]=10 and sz >= 3
			then
                            b := FSZtestZ(G);
                            if IsList(b) then return b;
                            else return true;
                            fi;
			fi;
		elif l[1]=3 then
			if l[2]<8 or sz < 4
                         or Exponent(G)<27
				then return true;
			elif l[2]=8 and sz>=4
			then
                            b := FSZtestZ(G);
                            if IsList(b) then return b;
                            else return true;
                            fi;
			fi;
		elif sz < l[1]+1
			then return true;
		elif sz = l[1]+1 and sz=l[2]
		then
                    b := FSZtestZ(G);
                    if IsList(b) then return b;
                    else return true;
                    fi;
		fi;
	else
        #check the exponent for non-p-groups
		l := FactorsInt(Exponent(G));
		p2 := Length(Positions(l,2));
		p3 := Length(Positions(l,3));
		po := Filtered(l,x->x>3);

		if ForAll(Collected(po),x->x[2]<2) and
			( (p2 < 4 and p3 < 4)
                    or ( p2 < 6 and p3 < 2) )  	
			then return true;
		fi;		
	fi;
	
	#tests were inconclusive
	return fail ;
end;
\end{lstlisting}

We then incorporate these changes into a modified version of \verb"FSZtest", which we give the same name.  Note that this function also uses the function \verb"beta" and its corresponding helper functions from \citep{PS16}.  It has the same inputs and outputs as \verb"FSZtestZ", except that the test is definitive, and so returns \verb"true" when the group is $FSZ$.

\begin{lstlisting}
FSZtest := function (G)
local C, CT, zz, z, cl, div, d, chi, m, b;

b := IMMtests(G);;

if not b=fail
	then return b;
fi;

cl := RationalClasses(G);
cl := Filtered(cl,c->not Order(Representative(c))
            in [1,2,3,4,6]);

for zz in cl do
        z := Representative(zz);
	C := Centralizer(G, z);

	div := Filtered(DivisorsInt(Exponent(C)/Order(z)),
                m->not Gcd(m,Order(z)) in [1,2,3,4,6]);
	
	if Length(div) < 1 then continue; fi;

        # Check for the easy cases	
	b := IMMtests(C);

	if b=true
		then continue;
	elif IsList(b) then
		if RationalClass(C,z)=RationalClass(C,b[2])
			then return b;
		fi;
	fi;

	CT := OrdinaryCharacterTable(C) ;

	for chi in Irr(CT) do
		for m in div do
			if not IsRat(beta(CT, z, m, chi))
                            then return [m,z,chi,CT];
			fi ;
		od;
	od;
od;

return true ;

end ;
\end{lstlisting}

Our typical procedure will be as follows: given a group $G$, take its Sylow 5-subgroup $P$ and find $u,g\in P$ such that $|P_5(u,g)|\neq |P_5(u,g^n)|$ for $5\nmid n$, and then show that $|G_5(u,g)|\neq |G_5(u,g^n)|$.  The second entry in the list returned by \verb"FSZtest" gives precisely the $g$ value we need.  But it does not provide the $u$ value directly, nor the $n$.  As it turns out, we can always take $n=2$ when $o(g)=5$, but for other orders this need not necessarily hold.

In order to acquire these values we introduce the function \verb"FSIndPt" below, which is a variation on \verb"FSInd" \citep{IMM}.  This function has the same essential limitation that \verb"FSInd" does, in that it needs to completely enumerate, store, and sort the elements of the group.  This could in principle be avoided, at the cost of increased run-time.  However our main use for the function is to apply it to Sylow 5-subgroups which have small enough order that this issue does not pop up.

The inputs are a group $G$, $m\in\BN$ and $g\in G$.  It is best if one in fact passes in $C_G(g)$ for $G$, but the function will compute the centralizer regardless.  The function looks for an element $u\in C_G(g)$ and an integer $j$ coprime to the order of $g$ such that $|G_m(u,g)| \neq |G_m(u,g^j)|$.  The output is the two element list \verb"[u,j]" if such data exists, otherwise it returns \verb"fail" to indicate that the test is normally inconclusive.  Note that by \cref{cor:u-g-exchange} and centrality of $g$ in $C=C_G(g)$ we need only consider the rational classes in $C$ to find such a $u$.
\begin{lstlisting}
FSIndPt:=function(G,m,g)

local GG, C, Cl, gucoeff, elG, Gm, alist,
        aulist, umlist, npos, j, n, u, pr;

C := Centralizer(G,g);

GG := EnumeratorSorted(C);;
elG := Size(C);

Gm := List(GG,x->Position(GG,x^m));

pr := PrimeResidues(Order(g));

for Cl in RationalClasses(C) do
	
    u := Representative(Cl);

    npos := [];
    alist := [];
    aulist := [];
    umlist := [];
    gucoeff := [];
	
    umlist := List(GG, a->Position(GG,
                (a*Inverse(u))^m));;

    #The following computes the cardinalities
    # of G_m(u,g^n).
    for n in pr
    do
    	npos := Position(GG, g^n);
    	alist := Positions(Gm, npos);
    	aulist := Positions(umlist, npos);
    	gucoeff[n] := Size(
                    Intersection(alist, aulist));

        #Check if we've found our u
        if not gucoeff[n] = gucoeff[1]
                then return [u,n];
        fi;
    od;
od;

#No u was found for this G,m,g
return fail;

end;
\end{lstlisting}

Lastly, we introduce the function \verb"FSZSetCards", which is the most naive and straightforward way of computing both $|G_m(u,g)|$ and $|G_m(u,g^n)|$.  The inputs are a set $C$ of group elements---normally this would be $C_G(g)$, but could be a conjugacy class or some other subset or subgroup---; group elements $u,g$; and integers $m,n$ such that $g\neq g^n$.  The output is a two element list, which counts the number of elements of $C$ in $G_m(u,g)$ in the first entry and the number of elements of $C$ in $G_m(u,g^n)$ in the second entry.  It is left to the user to check that the inputs satisfy whatever relations are needed, and to then properly interpret the output.
\begin{lstlisting}
FSZSetCards := function(C,u,g,m,n)
    local contribs, apow , aupow, a;

    contribs := [0,0];

    for a in C do
        apow := a^m;
        aupow := (a*Inverse(u))^m;

        if (apow = g and aupow = g) then
                contribs[1] := contribs[1]+1;

        elif (apow=g^n and aupow=g^n) then
                contribs[2] := contribs[2]+1;
        fi;
    od;

    return(contribs);
end;
\end{lstlisting}
As long as $C$ admits a reasonable iterator in GAP then this function can compute these cardinalities with a very minimal consumption of memory.  Any polycyclic or permutation group satisfies this, as well as any conjugacy class therein.  However, for a matrix group GAP will attempt to convert to a permutation representation, which is usually very costly.

The trade-off, as it often is, is in the speed of execution.  For permutation groups the run-time can be heavily impacted by the degree, such that it is almost always worthwhile to apply \verb"SmallerDegreePermutationRepresentation" whenever possible.  If the reader wishes to use this function on some group that hasn't been tested before, the author would advise adding in code that would give you some ability to gauge how far along the function is.  By default there is nothing in the above code, even if you interrupt the execution to check the local variables, to tell you if the calculation is close to completion.   Due to a variety of technical matters it is difficult to precisely benchmark the function, but when checking a large group it is advisable to acquire at least some sense of whether the calculation may require substantial amounts of time.
\begin{rem}
Should the reader opt to run our code to see the results for themselves, they may occasionally find that the outputs of \verb"FSZSetCards" occur in the opposite order we list here.  This is due to certain isomorphisms and presentations for groups calculated in GAP not always being guaranteed to be identical every single time you run the code.  As a result, the values for $u$ or $g$ may sometimes be a coprime power (often the inverse) of what they are in other executions of the code.  Nevertheless, there are no issues with the function proving the non-$FSZ$ property thanks to \cref{cor:u-g-exchange}, and there is sufficient predictability to make the order of the output the only variation.
\end{rem}
While very naive, \verb"FSZSetCards" will suffice for most of our purposes, with all uses of it completing in an hour or less.  However, in \cref{sec:symplectic} we will find an example where the expected run-time for this function is measured in weeks, and for which \verb"FSZtest" requires immense amounts of memory---\citet{PS16} says that \verb"FSZtest" for this group consumed 128 GB of memory without completing!

We therefore need a slightly less naive approach to achieve a more palatable run-time in this case.  We leave this to \cref{sec:symplectic}, but note to the reader that the method this section uses can also be applied to all of the other groups for which \verb"FSZSetCards" suffices.  The reason we bother to introduce and use \verb"FSZSetCards" is that the method of \cref{sec:symplectic} relies on being able to compute conjugacy classes, which can hit memory consumption issues that \verb"FSZSetCards" will not encounter.  It is not our goal with these functions to find the most efficient, general-purpose procedure. Instead we seek to highlight some of the ways in which computationally problematic groups may be rendered tractable by altering the approach one takes, and to show that the non-$FSZ$ property of these groups can be demonstrated in a (perhaps surprisingly) short amount of time and with very little memory consumption.

\section{The \texorpdfstring{non-$FSZ$}{non-FSZ} sporadic simple groups}\label{sec:simple}
The goal for this section is to show that the Chevalley group $G_2(5)$, and all sporadic simple groups with order divisible by $5^6$, as well as their Sylow 5-subgroups, are non-$FSZ_5$.  We begin with a discussion of the general idea for the approach.

Our first point of observation is that the only primes $p$ such that $p^{p+1}$ divides the order of any of these groups have $p\leq 5$.  Indeed, a careful analysis of the non-$FSZ$ groups of order $5^6$ found in \citep{IMM} shows that several of them are non-split extensions with a normal extra-special group of order $5^5$, which can be denoted in AtlasRep notation as $5^{1+4}.5$.  Consulting the known maximal subgroups for these groups we can easily infer that the Sylow 5-subgroups of $HN$, $G_2(5)$, $B$, and $Ly$ have this same form, and that the Monster has such a $p$-subgroup.  Indeed, $G_2(5)$ is a maximal subgroup of $HN$, and $B$ and $Ly$ have maximal subgroups containing a copy of $HN$, so these Sylow subgroups are all isomorphic.  Furthermore, the Monster's Sylow 5-subgroup has the form $5^{1+6}.5^2$, a non-split extension of the elementary abelian group of order 25 by an extra special group of order $5^7$ .  Given this, we suspect that these Sylow 5-subgroups are all non-$FSZ_5$, and that this will cause the groups themselves to be non-$FSZ_5$.

We can then exploit the fact that non-trivial $p$-groups all have non-trivial centers to obtain centralizers in the parent group that contain a Sylow 5-subgroup.  In the case of $G=HN$ or $G=G_2(5)$, we can quickly find $u,g\in P$, with $P$ a Sylow 5-subgroup of $G$, such that $|P_5(u,g)| \neq |P_5(u,g^2)|$, and show that for $H=C_{G}(g)$ we have $|H_5(u,g)|\neq |H_5(u,g^2)|$.  Since necessarily $|H_5(u,g)|=|G_5(u,g)|$ and $|H_5(u,g^2)|=|G_5(u,g^2)|$, this will show that $HN$ and $G_2(5)$ are non-$FSZ_5$.  Unfortunately, it turns out that $P$ is not normal in $H$ in either case, so the cardinalities of these sets in $H$ must be checked directly, rather than simply applying \cref{cor:normalSylow}.  The remaining groups require a little more work, for various reasons.

In the case of the Monster, there is a unique non-identity conjugacy class yielding a centralizer with order divisible by $5^9$. So we are free to pick any subgroup $G$ of $M$ that contains a centralizer with this same order.  Fortunately, not only is such a (maximal) subgroup known, but \citet{BrayWilson:M} have also computed a permutation representation for it. This is available in GAP via the \verb"AtlasRep" package.  This makes all necessary calculations for the Monster accessible. The Sylow 5-subgroup is fairly easily shown to be non-$FSZ_5$ directly.  However, the centralizer we get in this way has large order, and its Sylow 5-subgroup is not normal, making it impractical to work with on a personal computer.  However, further consultation of character tables shows that the Monster group has a unique conjugacy class of an element of order 10 whose centralizer is divisible by $5^6$.  So we may again pick any convenient (maximal) subgroup with such a centralizer, and it turns out the same maximal subgroup works.  We construct the appropriate element of order 10 by using suitable elements from Sylow subgroups of the larger centralizer, and similarly to get the element $u$.  Again it turns out that the Sylow 5-subgroup of this smaller subgroup is not normal, so we must compute the set cardinalities over the entire centralizer in question.  However, this centralizer is about 1/8000-th the size of the initial one, and we are subsequently able to calculate the appropriate cardinalities in under an hour.

The Baby Monster can then be handled by using the fact that the Monster contains the double cover of $B$ as the centralizer of an involution to obtain the centralizer we need in $B$ from a centralizer in $M$.  The author thanks Robert Wilson for reminding them of this fact.  For the Lyons group, the idea is much the same as for $HN$ and $G_2(5)$, with the additional complication that the \verb"AtlasRep" package does not currently contain any permutation representations for $Ly$.  To resolve this, we obtain a permutation representation for $Ly$, either computed directly in GAP or downloaded \citep{LyPerms}.  This is then used to construct a suitable permutation representation of the maximal subgroup in question.  Once this is done the calculations proceed without difficulties.

These calculations all make extensive use of the functions given in \cref{sec:functions}.

\subsection{Chevalley group \texorpdfstring{$G_2(5)$}{G2(5)}}\label{sub:chev}
We now show that $G_2(5)$ and its Sylow 5-subgroups are not $FSZ_5$.  This was independently verified in \citep{PS16}.  Since $G_2(5)$ is of relatively small order, it can be attacked quickly and easily.

\begin{thm}\label{thm:chev}
    The simple Chevalley group $G_2(5)$ and its Sylow 5-subgroup are non-$FSZ_5$.
\end{thm}
\begin{proof}
The claims follow from running the following GAP code.
  \begin{lstlisting}
    G := AtlasGroup("G2(5)");;
    P := SylowSubgroup(G,5);;

    # The following shows P is not FSZ_5
    g := FSZtestZ(P)[2];

    # Find u
    u := FSIndPt(P,5,g)[1];;
    C := Centralizer(G,g);;

    #Check the cardinalities
    FSZSetCards(C,u,g,5,2);
  \end{lstlisting}
  The output is [0,625], so it follows that $G$ and $P$ are both non-$FSZ_5$ as desired.
\end{proof}
We note that $P$ is not normal in $C$, and indeed $C$ is a perfect group of order $\numprint{375000} = 2^3\cdot 3\cdot 5^6$.

The call to \verb"FSZSetCards" above runs in approximately 11 seconds, which is approximately the amount of time necessary to run \verb"FSZtest" on $G_2(5)$ directly.  In this case, the use of \verb"FSZSetCards" is not particularly efficient, as the groups in question are of reasonably small sizes and permutation degree.  Nevertheless, this demonstrates the basic method we will employ for all subsequent groups.

\subsection{The Harada-Norton group}\label{sub:HN}
For the group $HN$ the idea proceeds similarly as for $G_2(5)$.

\begin{thm}\label{thm:HN}
  The Harada-Norton simple group $HN$ and its Sylow 5-subgroup are not $FSZ_5$.
\end{thm}
\begin{proof}
  To establish the claims it suffices to run the following GAP code.
  \begin{lstlisting}
    G := AtlasGroup("HN");;
    P := SylowSubgroup(G,5);;

    # G, thus P, has very large degree.
    # Polycyclic groups are easier to work with.
    isoP := IsomorphismPcGroup(P);;
    P := Image(isoP);;

    #Find u, g for P
    g := FSZtestZ(P)[2];
    u := FSIndPt(P,5,g)[1];

    g := Image(InverseGeneralMapping(isoP),g);;
    u := Image(InverseGeneralMapping(isoP),u);;

    C := Centralizer(G,g);;
    isoC := IsomorphismPcGroup(C);;
    C := Image(isoC);;

    FSZSetCards(C,Image(isoC,u),Image(isoC,g),5,2);
  \end{lstlisting}
  This code executes in approximately 42 minutes, with approximately 40 of that spent finding $P$.  The final output is [3125,0], so we conclude that both $P$ and $HN$ are non-$FSZ_5$, as desired.
\end{proof}
$P$ is again not a normal subgroup of $C$, so we again must test the entire centralizer rather than just $P$.  We note that $|C|=2^5 5^6=\numprint{500000}$.  Indeed, $C$ is itself non-$FSZ_5$ of necessity, and the fact that the call to \verb"IsomorphismPcGroup" did not fail means that $C$ is solvable, and in particular not perfect and not a $p$-group.

\subsection{The Monster group}\label{sub:monster}
We will now consider the Monster group $M$.  The full Monster group is famously difficult to compute in.  But, as detailed in the beginning of the section, by consulting character tables of $M$ and its known maximal subgroups, we can find a maximal subgroup which contains a suitable centralizer (indeed, two suitable centralizers) and also admits a known permutation representation \citep{BrayWilson:M}.

\begin{thm}\label{thm:Monster}
  The Monster group $M$ and its Sylow 5-subgroup are not $FSZ_5$.
\end{thm}
\begin{proof}
  The Sylow 5-subgroup of $M$ has order $5^9$.  Consulting the character table of $M$, we see that $M$ has a unique conjugacy class yielding a proper centralizer with order divisible by $5^9$, and a unique conjugacy class of an element of order 10 whose centralizer has order divisible by $5^6$; moreover, the order of the latter centralizer is precisely 12 million, and in particular is not divisible by $5^7$.  It suffices to consider any maximal subgroups containing such centralizers.  The maximal subgroup of shape $5_{+}^{1+6}:2.J_2.4$, which is the normalizer associated to a $5B$ class, is one such choice.

  We first show that the Sylow 5-subgroup of $M$ is not $FSZ_5$.

  \begin{lstlisting}
    G := AtlasGroup("5^(1+6):2.J2.4");;
    P := SylowSubgroup(G,5);;
    isoP := IsomorphismPcGroup(P);;
    P := Image(isoP);;

    ex := FSZtestZ(P);
  \end{lstlisting}
  The proper centralizer with order divisible by $5^9$ is still impractical to work with.  So we will use the data for $P$ to construct the element of order 10 mentioned above.

  \begin{lstlisting}
    zp := ex[2];;
    zp := Image(InverseGeneralMapping(isoP),zp);;
    C := Centralizer(G,zp);;
    Q := SylowSubgroup(C,2);;

    zq := First(Center(Q),q->Order(q)>1 and
                Size(Centralizer(G,zp*q))=12000000);;

    #This gives us the g and centralizer we want.
    g := zp*zq;;
    C := Centralizer(G,g);;

    #Reducing the permutation degree will
    #save a lot of computation time later.
    isoC := SmallerDegreePermutationRepresentation(C);;
    C := Image(isoC);;
    g := Image(isoC,g);;
    zp := Image(isoC,zp);;
    zq := Image(isoC,zq);;

    #Now proceed to construct a choice of u.
    P := SylowSubgroup(C,5);;
    isoP := IsomorphismPcGroup(P);;
    P := Image(isoP);;

    ex := FSIndPt(P,5,Image(isoP,zp));

    up := Image(InverseGeneralMapping(isoP),ex[1]);;

    #Define our choice of u.
    #In this case, u has order 50.
    u := up*zq;;

    #Finally, we compute the cardinalities
    # of the relevant sets.
    FSZSetCards(C,u,g,5,7);
  \end{lstlisting}
  This final function yields [0,15000], which proves that $M$ is not $FSZ_5$, as desired.
\end{proof}
This final function call takes approximately 53 minutes to complete, while all preceding operations can complete in about 5 minutes combined---though the conversion of $C$ to a lower degree may take more than this, depending.  The lower degree $C$ has degree $18,125$, but requires (slightly) more than 2 GB of memory to acquire.  This conversion can be skipped to keep the memory demands well under 2GB, but the execution time for \verb"FSZSetCards" will inflate to approximately a day and a half.
\begin{rem}
  In the first definition of $C$ above, containing the full Sylow 5-subgroup of $M$, we have $|C|=9.45\times 10^{10} = 2^8\cdot 3^3\cdot 5^9\cdot 7$.  For the second definition of $C$, corresponding to the centralizer of an element of order 10, we have $|C|=1.2\times 10^7= 2^8\cdot 3\cdot 5^6$.  The first centralizer is thus $7875 = 3^2\cdot 5^3\cdot 7$ times larger than the second one.  Either one is many orders of magnitude smaller than $|M|\approx  8.1\times 10^{53}$, but the larger one was still too large to work with for practical purposes.
\end{rem}

\subsection{The Baby Monster}\label{sub:baby}
We can now consider the Baby Monster $B$.

\begin{thm}
  The Baby Monster $B$ and its Sylow 5-subgroup are both non-$FSZ_5$.
\end{thm}
\begin{proof}
  The Baby Monster is well known to have a maximal subgroup of the form $HN.2$, so it follows that $B$ and $HN$ have isomorphic Sylow 5-subgroups.  By \cref{thm:HN} $HN$ has a non-$FSZ_5$ Sylow 5-subgroup, so this immediately gives the claim about the Sylow 5-subgroup of $B$.

  From the character table of $B$ we see that there is a unique non-identity conjugacy class whose centralizer has order divisible by $5^6$.  This corresponds to an element of order 5 from the 5B class, and the centralizer has order $\numprint{6000000}= 2^7\cdot 3\cdot 5^6$.  In the double cover $2.B$ of $B$, this centralizer is covered by the centralizer of an element of order 10. This centralizer necessarily has order $12,000,000$.  Since $M$ contains $2.B$ as a maximal subgroup, and there is a unique centralizer of an element of order 10 in $M$ with order divisible by $\numprint{12000000}$, these centralizers in $2.B$ and $M$ are isomorphic.  We have already computed this centralizer in $M$ in \cref{thm:Monster}.  To obtain the centralizer in $B$, we need only quotient by an appropriate central involution.  In the notation of the proof of \cref{thm:Monster}, this involution is precisely \verb"zq".

  GAP will automatically convert this quotient group $D$ into a lower degree representation, yielding a permutation representation of degree 3125 for the centralizer.  This will require as much as 8GB of memory to complete.  Moreover, the image of \verb"zp" from \cref{thm:Monster} in this quotient group yields the representative of the 5B class we desire, denoted here by $g$.  Using the image of \verb"up" in the quotient for $u$, we can then easily run \verb"FSZSetCards(C,u,g,5,2)" to get a result of [15000,3125], which shows that $B$ is non-$FSZ_5$ as desired.  This final call completes in about 4 minutes.
\end{proof}
Note that in $M$ the final return values summed to \numprint{15000}, with one of the values 0, whereas in $B$ they sum to \numprint{18125} and neither is zero.  This reflects how there is no clear relationship between the $FSZ$ properties of a group and its quotients, even when the quotient is by a (cyclic) central subgroup.  In particular, it does not immediately follow that the quotient centralizer would yield the non-$FSZ$ property simply because the centralizer in $M$ did, or vice versa.

Moreover, we also observe that the cardinalities computed in \cref{thm:HN} implies that for a Sylow 5-subgroup $P$ of $B$ we have $P_5(u,g^2)=\emptyset$, so the 3125 "extra" elements obtained in $B_5(u,g^2)$ come from non-trivial conjugates of $P$.  This underscores the expected difficulties in a potential proof (or disproof) of \cref{conj}.

\subsection{The Lyons group}\label{sub:lyons}
There is exactly one other sporadic group with order divisible by $5^6$ (or $p^{p+1}$ for $p>3$): the Lyons group $Ly$.

\begin{thm}\label{thm:lyons}
  The maximal subgroup of $Ly$ of the form $5^{1+4}:4.S_6$ has a faithful permutation representation on 3,125 points, given by the action on the cosets of $4.S_6$.  Moreover, this maximal subgroup, $Ly$, and their Sylow 5-subgroups are all non-$FSZ_5$.
\end{thm}
\begin{proof}
  It is well-known that $Ly$ contains a copy of $G_2(5)$ as a maximal subgroup, and that the order of $Ly$ is not divisible by $5^7$.  Therefore $Ly$ and $G_2(5)$ have isomorphic Sylow 5-subgroups, and by \cref{thm:chev} this Sylow subgroup is not $FSZ_5$.

  Checking the character table for $Ly$ as before, we find there is a unique non-identity conjugacy class whose corresponding centralizer has order divisible by $5^6$.  In particular, the order of this centralizer is $\numprint{2250000}=2^4\cdot 3^2\cdot 5^6$, and it comes from an element of order $5$.  So any maximal subgroup containing an element of order 5 whose centralizer has this order will suffice.  The maximal subgroup $5^{1+4}:4.S_6$ is the unique such choice.

  The new difficulty here is that, by default, there are only matrix group representations available though the \verb"AtlasRep" package for $Ly$ and $5^{1+4}:4.S_6$, which are ill-suited for our purposes.  However, faithful permutation representations for $Ly$ are known, and they can be constructed through GAP with sufficient memory available provided one uses a well-chosen method.  A detailed description of how to acquire the permutation representation on \numprint{8835156} points, as well as downloads for the generators (including MeatAxe versions courtesy of Thomas Breuer) can be found on the web, courtesy \citet{LyPerms}.

  Using this, we can then obtain a permutation representation for the maximal subgroup $5^{1+4}:4.S_6$ on \numprint{8835156} points using the programs available on the online ATLAS \citep{OnlineAtlas}.  This in turn is fairly easily converted into a permutation representation on a much smaller number of points, provided one has up to 8 GB of memory available, via \verb"SmallerDegreePermutationRepresentation".  The author obtained a permutation representation on $3125$ points, corresponding to the action on the cosets of $4.S_6$.  The exact description of the generators is fairly long, so we will not reproduce them here.  The author is happy to provide them upon request.  One can also proceed in a fashion similar to some of the cases handled in \citep{BrayWilson:M} to find such a permutation representation.

  Once this smaller degree representation is obtained, it is then easy to apply the same methods as before to show the desired claims about the $FSZ_5$ properties.  We can directly compute the Sylow 5-subgroup, then find $u,g$ through \verb"FSZtestZ" and \verb"FSIndPt" irrespectively, set $C$ to be the centralizer of $g$, then run \verb"FSZSetCards(C,u,g,5,2)".  This returns [5000,625], which gives the desired non-$FSZ_5$ claims.
\end{proof}
Indeed, \verb"FSZtest" can be applied to (both) the centralizer and the maximal subgroup once this permutation representation is obtained.  This will complete quickly, thanks to the relatively low orders and degrees involved. We also note that the centralizer $C$ so obtained will not have a normal Sylow 5-subgroup, and is a perfect group.  The maximal subgroup in question is neither perfect nor solvable, and does not have a normal Sylow 5-subgroup.

\section{The \texorpdfstring{$FSZ$}{FSZ} sporadic simple groups}\label{sec:fischer}
We can now show that all other sporadic simple groups and their Sylow subgroups are $FSZ$.
\begin{example}
  Any group which is necessarily $FSZ$ (indeed, $FSZ^+$) by \citep[Corollary 5.3]{IMM} necessarily has all of its Sylow subgroups $FSZ$, and so satisfies the conjecture.  This implies that all of the following sporadic groups, as well as their Sylow $p$-subgroups, are $FSZ$ (indeed, $FSZ^+$).
  \begin{itemize}
    \item The Mathieu groups $M_{11},M_{12},M_{22},M_{23},M_{24}$.
    \item The Janko groups $J_1,J_2,J_3,J_4$.
    \item The Higman-Simms group $HS$.
    \item The McLaughlin group $McL$.
    \item The Held group $He$.
    \item The Rudvalis group $Ru$.
    \item The Suzuki group $Suz$.
    \item The O'Nan group $O'N$.
    \item The Conway group $Co_3$.
    \item The Thompson group $Th$.
    \item The Tits group ${}^2 F_4(2)'$.
  \end{itemize}
\end{example}
\begin{example}
  Continuing the last example, it follows that the following are the only sporadic simple groups not immediately in compliance with the conjecture thanks to \citep[Corollary 5.3]{IMM}.
  \begin{itemize}
    \item The Conway groups $Co_1,Co_2$.
    \item The Fischer groups $Fi_{22},Fi_{23},Fi_{24}'$.
    \item The Monster $M$.
    \item The Baby Monster $B$.
    \item The Lyons group $Ly$.
    \item The Harada-Norton group $HN$.
  \end{itemize}
  The previous section showed that the last four groups were all non-$FSZ_5$ and have non-$FSZ_5$ Sylow 5-subgroups, and so conform to the conjecture.  By exponent considerations the Sylow subgroups of the Conway and Fischer groups are all $FSZ^+$.  The function \verb"FSZtest" can be used to quickly show that $Co_1$, $Co_2$, $Fi_{22}$, and $Fi_{23}$ are $FSZ$, and so conform to the conjecture.
\end{example}
This leaves just the largest Fischer group $Fi_{24}'$.
\begin{thm}\label{thm:fi24}
    The sporadic simple group $Fi_{24}'$ and its Sylow subgroups are all $FSZ$.
\end{thm}
\begin{proof}
  The exponent of $Fi_{24}'$ can be calculated from its character table and shown to be
  \[ \numprint{24516732240} = 2^4 \cdot 3^3 \cdot 5\cdot 7 \cdot 11 \cdot 13\cdot 17\cdot 23\cdot 29.\]
  As previously remarked, this automatically implies that the Sylow subgroups are all $FSZ$ (indeed, $FSZ^+$).  By \citep[Corollary 5.3]{IMM} it suffices to show that every centralizer of an element with order not in $\{1,2,3,4,6\}$ in $Fi_{24}'$ that contains an element of order 16 is $FSZ$.  There is a unique conjugacy class in $Fi_{24}'$ for an element with order (divisible by) 16.  The centralizer of such an element has order 32, and is isomorphic to $\BZ_{16}\times \BZ_2$.  So it suffices to consider the elements of order 8 in this centralizer, and show that their centralizers (in $Fi_{24}'$) are $FSZ$.  Every such element has a centralizer of order $1536=2^9\cdot 3$.  So by \cref{thm:small-order} the result follows.

  The following is GAP code verifying these claims.
  \begin{lstlisting}
    G := AtlasGroup("Fi24'");;
    GT := CharacterTable("Fi24'");;

    Positions(OrdersClassRepresentatives(GT) mod 16,0);

    exp := Lcm(OrdersClassRepresentatives(GT));
    Collected(FactorsInt(exp));
    SetExponent(G,exp);;

    P := SylowSubgroup(G,2);;

    #There are many ways to get an element of order 16.
    #Here's a very crude, if non-deterministic, one.
    x := Random(P);;
    while not Order(x) = 16 do x:=Random(P); od;

    C := Centralizer(G,x);;

    cents := Filtered(C,y->Order(y)=8);;
    cents := List(cents,y->Centralizer(G,y));;

    List(cents,Size);
  \end{lstlisting}
\end{proof}

The following then summarizes our results on sporadic simple groups.
\begin{thm}\label{thm:summary}
  The following are equivalent for a sporadic simple group $G$.
  \begin{enumerate}
    \item $G$ is not $FSZ$.
    \item $G$ is not $FSZ_5$.
    \item The order of $G$ is divisible by $5^6$.
    \item $G$ has a non-$FSZ$ Sylow subgroup.
    \item The Sylow 5-subgroup of $G$ is not $FSZ_5$.
  \end{enumerate}
\end{thm}
\begin{proof}
  Combine the results of this section and the previous one.
\end{proof}

\section{The symplectic group \texorpdfstring{\ensuremath{S_6(5)}}{PSp(6,5)}}\label{sec:symplectic}
In \citep{PS16} it was mentioned that the symplectic group $S_6(5)$ was likely to be the second smallest non-$FSZ$ simple group, after $G_2(5)$.  Computer calculations there ran into issues when checking a particular centralizer, as the character table needed excessive amounts of memory to compute.  Our methods so far also place this group at the extreme end of what's reasonable.  In principle the procedure and functions we've introduced so far can decide that this group is non-$FSZ$ in an estimated two weeks of uninterrupted computations, and with nominal memory usage.  However, we can achieve a substantial improvement that completes the task in about 8 hours (on two processes; 16 hours for a single process), while maintaining nominal memory usage.

The simple yet critical observation comes from \citep[Definition 3.3]{PS16}.  In particular, if $a\in G_m(u,g)$, then $a^m=g$ implies that for all $b\in \class_{C_G(g)}(a)$ we have $b^m=g$.  So while \verb"FSZSetCards" acts as naively as possible and iterates over all elements of $C=C_G(g)$, we in fact need to only iterate over the elements of those conjugacy classes of $C$ whose $m$-th power is $g$ (or $g^n$).  GAP can often compute the conjugacy classes of a finite permutation or polycyclic group quickly and efficiently.  So while it is plausible that finding these conjugacy classes can be too memory intensive for certain centralizers, there will nevertheless be centralizers for which all other methods are too impractical for either time or memory reasons, but for which this reduction to conjugacy classes makes both time and memory consumption a non-issue.  The otherwise problematic centralizer of $S_6(5)$ is precisely such a case, as we will now see.

\begin{thm}\label{thm:symplectic}
  The projective symplectic group $S_6(5)$ and its Sylow $5$-subgroup are both non-$FSZ_5$.
\end{thm}
\begin{proof}
As usual, our first task is to show that the Sylow $5$-subgroup is non-$FSZ_5$, and then use the data obtained from that to attack $S_6(5)$.
  \begin{lstlisting}
    G := AtlasGroup("S6(5)");;
    P := SylowSubgroup(G,5);;
    isoP := IsomorphismPcGroup(P);;
    P := Image(isoP);;

    #Show P is non-FSZ_5, and
    #get the g we need via FSZtestZ
    g := FSZtestZ(P)[2];

    #Get the u we need via FSIndPt
    u := FSIndPt(P,5,g)[1] ;
  \end{lstlisting}
  One can of course store the results of \verb"FSZtestZ" and \verb"FSIndPt" directly to see the complete data returned, and then extract the specific data need.

  We can then show that $G=S_6(5)$ is itself non-$FSZ_5$ by computing $G_5(u,g)$ and $G_5(u,g^2)$ with the following code.
  \begin{lstlisting}
    G := Centralizer(G,g);;
    isoG := SmallerDegreePermutationRepresentation(G);;
    G := Image(isoG);;
    g := Image(isoG,g);;
    u := Image(isoG,u);;
    uinv := Inverse(u);;

    #Now we compute the conjugacy classes
    # of the centralizer.
    cl := ConjugacyClasses(G);;

    #We then need only consider those
    # classes with a suitable 5-th power
    cand1 := Filtered(cl,x->Representative(x)^5=g);;
    cand2 := Filtered(cl,x->Representative(x)^5=g^2);;

    #There is in fact only one conjugacy
    # class in both cases.
    Length(cand1);
    Length(cand2);
    cand1 := cand1[1];;
    cand2 := cand2[1];;

    #The following computes |G_5(u,g)|
    Number(cand1,x->(x*uinv)^5=g);

    #The following computes |G_5(u,g^2)|
    Number(cand2,x->(x*uinv)^5=g^2);
  \end{lstlisting}
  This code shows that \begin{align*} |G_5(u,g)|&=\numprint{1875000};\\|G_5(u,g^2)|&=\numprint{375000}.\end{align*}
  Therefore $S_6(5)$ is non-$FSZ_5$, as desired.
\end{proof}
The calculation of $|G_5(u,g)|$ takes approximately 8.1 hours, and the calculation of $|G_5(u,g^2)|$ takes approximately 7.45 hours.  The remaining calculations are done in significantly less combined time.  We note that the calculations of these two cardinalities can be done independently, allowing each one to be calculated simultaneously on separate GAP processes.

We also note that the centralizer in $S_6(5)$ under consideration in the above is itself a perfect group; is a permutation group of degree 3125 and order 29.25 billion; and has a non-normal Sylow 5-subgroup.  Moreover, it can be shown that the $g$ we found yields the only rational class of $P$ at which $P$ fails to be $FSZ$.  One consequence of this, combined with the character table of $S_6(5)$, is that, unlike in the case of the Monster group, we are unable to switch to any other centralizer with a smaller Sylow 5-subgroup to demonstrate the non-$FSZ_5$ property.

Similarly as with the Baby Monster group, it is interesting to note that $|P_5(u,g)|=\numprint{62500}$ and $|P_5(u,g^2)|=0$ for $P, u, g$ as in the proof.  These cardinalities can be quickly computed exactly as they were for $S_6(5)$, simply restricted to $P$, or using the slower \verb"FSZSetCards", with the primary difference being that now there are multiple conjugacy classes to check and sum over.

Before continuing on to the next section, where we consider small order perfect groups available in GAP, we wish to note a curious dead-end, of sorts.

\begin{lem}
Given $u,g\in G$ with $[u,g]=1$, let $C=C_G(g)$, $D=C_C(u)$, and $m\in\BN$.  Then $a\in G_m(u,g)$ if and only if $a^d\in G_m(u,g)$ for some/any $d\in D$.
\end{lem}
\begin{proof}
This is noted by \citet{IMM} when introducing the concept of an $FSZ^+$ group.  It is an elementary consequence of the fact that $D=C_G(u,g)$ centralizes both $g$ and $u$ by definition.
\end{proof}

So suppose we have calculated those conjugacy classes in $C$ whose $m$-th power is $g$.  As in the above code, we can iterate over all elements of these conjugacy classes in order to compute $|G_m(u,g)|$.  However, the preceding lemma shows that we could instead partition each such conjugacy class into orbits under the $D$ action.  The practical upshot then being that we need only consider a single element of each orbit in order to compute $|G_m(u,g)|$.

In the specific case of the preceding theorem, we can show that the single conjugacy classes \verb"cand1" and \verb"cand2" both have precisely 234 million elements, and that $D$ is a non-abelian group of order \numprint{75000}, and is in fact the full centralizer of $u$ in $S_6(5)$.  Moreover the center of $C$ is generated by $g$, and so has order $5$.  Thus in the best-case scenario partitioning these conjugacy classes into $D$ orbits can result in orbits with $|D/Z(C)|=\numprint{15000}$ elements each.  The cardinalities we computed can also be observed to be multiples \numprint{15000}.  That would constitute a reduction of more than four orders of magnitude on the total number of elements we would need to check.  While this is a best-case scenario, since $D$ also has index \numprint{390000} in $C$ it seems very plausible that such a partition would produce a substantial reduction in the number of elements to be checked.  So provided that calculating these orbits can be done reasonably quickly, we would expect a significant reduction in run-time.

There is a practical problem, however.  The problem being that, as far as the author can tell, there is no efficient way for GAP to actually compute this partition.  Doing so evidently requires that GAP fully enumerate and store the conjugacy class in question.  In our particular case, a conjugacy class of 234 million elements in a permutation group of degree 3125 simply requires far too much memory---in excess of 1.5 terabytes.  As such, while the lemma sounds promising, it seems to be lacking in significant practical use for computer calculations.  It seems likely, in the author's mind, that any situation in which it is useful could have been handled in reasonable time and memory by other methods. Nevertheless, the author cannot rule out the idea as a useful tool.

\section{Perfect groups of order less than \texorpdfstring{$10^6$}{1,000,000}}\label{sec:perfect}
We now look for examples of additional non-$FSZ$ perfect groups.  The library of perfect groups stored by GAP has most perfect groups of order less than $10^6$, with a few exceptions noted in the documentation.  So we can iterate through the available groups, of which there are 1097 at the time this paper was written.  We can use the function \verb"IMMtests" from \cref{sec:functions} to show that most of them are $FSZ$.

\begin{lstlisting}
    #Get all available sizes
    Glist := Filtered(SizesPerfectGroups(),
                n->NrPerfectLibraryGroups(n)>0);;

    #Get all available perfect groups
    Glist := List(Glist,
        n->List([1..NrPerfectLibraryGroups(n)],
        k->PerfectGroup(IsPermGroup, n, k)));;

    Glist := Flat(Glist);;

    #Remove the obviously FSZ ones
    Flist := Filtered(Glist, G->not IMMtests(G)=true);;
\end{lstlisting}
This gives a list of 63 perfect groups which are not immediately dismissed as being $FSZ$.

\begin{thm}
  Of the 1097 perfect groups of order less than $10^6$ available through the GAP perfect groups library, exactly 7 of them are not $FSZ$, all of which are extensions of $A_5$.  All seven of them are non-$FSZ_5$.  Four of them have order $\numprint{375000} = 2^3\cdot 3 \cdot 5^6$, and three of them have order $\numprint{937500} = 2^2\cdot 3 \cdot 5^7$.  Their perfect group ids in the library are:

  \begin{align*}
    [375000,2],&& [375000,8],&& [375000,9],&&
    [375000,11],\\
    [937500,3],&&
    [937500,4],&&
    [937500,5]\hphantom{,}&&
  \end{align*}
\end{thm}
\begin{proof}
  Continuing the preceding discussion, we can apply \verb"FSZtest" to the 63 groups in \verb"Flist" to obtain the desired result.  This calculation takes approximately two days of total calculation time on the author's computer, but can be easily split across multiple GAP instances.  Most of the time is spent on the $FSZ$ groups of orders \numprint{375000} and \numprint{937500}.
\end{proof}

On the other hand, we can also consider the Sylow subgroups of all 1097 available perfect groups, and test them for the $FSZ$ property.

\begin{thm}
  If $G$ is one of the 1097 perfect groups of order less than $10^6$ available through the GAP perfect groups library, then the following are equivalent.
  \begin{enumerate}
    \item $G$ is not $FSZ$.
    \item $G$ has a non-$FSZ$ Sylow subgroup.
    \item $G$ has a non-$FSZ_5$ Sylow 5-subgroup.
    \item $G$ is not $FSZ_5$.
  \end{enumerate}
\end{thm}
\begin{proof}
  Most of the GAP calculations we need to perform now are quick, and the problem is easily broken up into pieces, should it prove difficult to compute everything at once.  The most memory intensive case requires about 1.7 GB to test.  With significantly more memory available than this, the cases can simply be tested by \verb"FSZtest" en masse, which will establish the result relatively quickly---a matter of hours.  We sketch the details here and leave it to the interested reader to construct the relevant code.  Recall that it is generally worthwhile to convert $p$-groups into polycyclic groups in GAP via \verb"IsomorphismPcGroup".

  Let \verb"Glist" be constructed in GAP as before.  Running over each perfect group, we can easily construct their Sylow subgroups.  We can then use \verb"IMMtests" from \cref{sec:functions} to eliminate most cases. There are 256 Sylow subgroups, each from a distinct perfect group, for which \verb"IMMtests" is inconclusive; and there are exactly 4 cases where \verb"IMMtests" definitively shows the non-$FSZ$ property, which are precisely the Sylow 5-subgroups of each of the non-$FSZ$ perfect groups of order \numprint{375000}.  These 4 Sylow subgroups are all non-$FSZ_5$.  We can also apply \verb"FSZtestZ" to the Sylow 5-subgroups of the non-$FSZ$ perfect groups of order \numprint{937500} to conclude that they are all non-$FSZ_5$.  All other Sylow subgroups remaining that come from a perfect group of order less than 937,500 can be shown to be $FSZ$ by applying \verb"FSZtest" without difficulty.  Of the three remaining Sylow subgroups, one has a direct factor of $\BZ_5$, and the other factor is easily tested and shown to be $FSZ$, whence this Sylow subgroup is $FSZ$.  This leaves two other cases, which are the Sylow 5-subgroups of the perfect groups with ids [937500,7] and [937500,8].  The second of these is easily shown to be FSZ by \verb"FSZtest".  The first can also be tested by \verb"FSZtest", but this is the case that requires the most memory and time---approximately 15 minutes and the indicated 1.7 GB.   In this case as well the Sylow subgroups are $FSZ$.  This completes the proof.
\end{proof}

\bibliographystyle{plainnat}
\bibliography{../references}

\end{document}